\documentclass[reqno]{amsart}
\usepackage{amssymb,latexsym}
\usepackage{amsmath}
\usepackage{amsthm}
\usepackage{graphicx}
\usepackage{esint}
\usepackage{titletoc}
\usepackage{palatino,mathpazo}
\numberwithin{equation}{section}
\newtheorem{theorem}{Theorem}[section]
\newtheorem{proposition}[theorem]{Proposition}
\newtheorem{lemma}[theorem]{Lemma}
\newtheorem{corollary}[theorem]{Corollary}
\newtheorem{remark}[theorem]{Remark}

\theoremstyle{definition}

\def\XXint#1#2#3{{\setbox0=\hbox{$#1{#2#3}{\int}$}
     \vcenter{\hbox{$#2#3$}}\kern-.5\wd0}}

\def\p{\partial}

\def\o{\Omega}
\def\u{\mathbf{u}}
\def\R{{\mathbb R}}
\def\S{{\mathbb S}}
\def\N{{\mathbb N}}

\renewcommand{\d}{\delta }
\newcommand{\e }{\varepsilon }
\newcommand{\ov}{\overline}
\newcommand{\n }{\nabla }
\renewcommand{\th }{\theta }

\begin{document}

\title[Wave equations associated to Liouville-type problems]{Wave equations associated to Liouville-type problems: global existence in time and blow up criteria}

\author{ Weiwei ~Ao}
\address{ Weiwei ~Ao,~Department of mathematics and statistics, Wuhan university, Wuhan, 430072, P.R. China }
\email{wwao@whu.edu.cn}

\author{ Aleks Jevnikar}
\address{ Aleks Jevnikar,~University of Rome `Tor Vergata', Via della Ricerca Scientifica 1, 00133 Roma, Italy}
\email{jevnikar@mat.uniroma2.it}

\author{ Wen Yang}
\address{ Wen ~Yang,~Wuhan Institute of Physics and Mathematics, Chinese Academy of Sciences, P.O. Box 71010, Wuhan 430071, P. R. China}
\email{math.yangwen@gmail.com}



\begin{abstract}
We are concerned with wave equations associated to some Liouville-type problems on compact surfaces, focusing on sinh-Gordon equation and general Toda systems. Our aim is on one side to develop the analysis for wave equations associated to the latter problems and second, to substantially refine the analysis initiated in \cite{cy} concerning the mean field equation. In particular, by exploiting the variational analysis recently derived for Liouville-type problems we prove global existence in time for the sub-critical case and we give general blow up criteria for the super-critical and critical case. The strategy is mainly based on fixed point arguments and improved versions of the Moser-Trudinger inequality.
\end{abstract}

\maketitle

\section{introduction}
In this paper we are concerned with wave equations associated to some Liouville-type problems on compact surfaces arising in mathematical physics: sinh-Gordon equation \eqref{w-sg} and some general Toda systems \eqref{w-toda}. The first wave equation we consider is
\begin{equation}
\label{w-sg}
\partial_t^2u-\Delta_gu=\rho_1\left(\frac{e^u}{\int_{M}e^u}-\frac{1}{|M|}\right)
-\rho_2\left(\frac{e^{-u}}{\int_{M}e^{-u}}-\frac{1}{|M|}\right) \quad \mbox{on } M,
\end{equation}
with $u:\R^+\times M\to \R$, where $(M,g)$ is a compact Riemann surface with total area $|M|$ and metric $g$, $\Delta_g$ is the Laplace-Beltrami operator and $\rho_1, \rho_2$ are two real parameters. Nonlinear evolution equations have been extensively studied in the literature due to their many applications in physics, biology, chemistry, geometry and so on. In particular, the sinh-Gordon model \eqref{w-sg} has been applied to a wide class of mathematical physics problems such as quantum field theories, non commutative field theories, fluid dynamics, kink dynamics, solid state physics, nonlinear optics and we refer to \cite{a, cm, ch, cho, mmr, nat, w, zbp} and the references therein.

\medskip

The stationary equation related to \eqref{w-sg} is the following sinh-Gordon equation:
\begin{equation}
\label{sg}
-\Delta_gu=\rho_1\left(\frac{e^u}{\int_{M}e^u}-\frac{1}{|M|}\right)
-\rho_2\left(\frac{e^{-u}}{\int_{M}e^{-u}}-\frac{1}{|M|}\right).
\end{equation}
In mathematical physics the latter equation describes the mean field equation of the equilibrium turbulence with arbitrarily signed vortices, see \cite{jm, pl}. For more discussions concerning the physical background we refer for example to \cite{c,l,mp,n,os} and the references therein. On the other hand, the case $\rho_{1}=\rho_{2}$ has a close relationship with constant mean curvature surfaces, see \cite{w1,w2}.

Observe that for $\rho_{2}=0$ equation (\ref{sg}) reduces to the following well-know mean field equation:
\begin{equation}
\label{mf}
-\Delta_g u=\rho\left(\frac{e^u}{\int_{M}e^u}-\frac{1}{|M|}\right),
\end{equation}
which has been extensively studied in the literature since it is related to the prescribed Gaussian curvature problem \cite{ba, sch} and Euler flows \cite{cag, kies}. There are by now many results concerning \eqref{mf} and we refer to the survey \cite{tar}. On the other hand, the wave equation associated to \eqref{mf} for $M=\S^2$, that is
\begin{equation}
\label{w-mf}
\partial_t^2u-\Delta_g u=\rho\left(\frac{e^u}{\int_{\S^2}e^u}-\frac{1}{4\pi}\right) \quad \mbox{on }  \S^2,
\end{equation}
was recently considered in \cite{cy}, where the authors obtained some existence results and a first blow up criterion. Let us focus for a moment on the blow up analysis. They showed that in the critical case $\rho=8\pi$ for the finite time blow up solutions to \eqref{w-mf} there exist a sequence $t_k\to T_0^{-}<+\infty$ and a point $x_1\in \S^2$ such that for any $\varepsilon>0,$
\begin{equation} \label{mf-cr}
\lim_{k\to+\infty}\frac{\int_{B(x_{1},\varepsilon)}e^{u(t_k,\cdot)}}{\int_{\S^2}e^{u(t_k,\cdot)}}\geq 1-\varepsilon,
\end{equation}
i.e. the measure $e^{u(t_k)}$ (after normalization) concentrates around one point on $\S^2$ (i.e. it resembles a one bubble). On the other hand, for the general super-critical case $\rho>8\pi$ the blow up analysis is not carried out and we are missing the blow up criteria. One of our aims is to substantially refine the latter analysis and to give general blow up criteria, see Corollary \ref{cr1.1}. As a matter of fact, this will follow by the analysis we will develop for more general problems: the sinh-Gordon equation \eqref{w-sg} and Toda systems \eqref{w-toda}.

\medskip

Let us return now to the sinh-Gordon equation \eqref{sg} and its associated wave equation \eqref{w-sg}. In the last decades the analysis concerning \eqref{mf} was generalized for treating the sinh-Gordon equation \eqref{sg} and we refer to \cite{ajy, jwy, jwy2, jwyz, os} for blow up analysis, to \cite{gjm} for uniqueness aspects and to \cite{bjmr, j1, j2, j3} for what concerns existence results. On the other hand, for what concerns the wave equation associated to \eqref{sg}, i.e. \eqref{w-sg}, there are few results mainly focusing on traveling wave solutions, see for example \cite{a, fll, nat, w, zbp}. One of our aims is to develop the analysis for \eqref{w-sg} in the spirit of \cite{cy} and to refine it with some new arguments. More precisely, by exploiting the variational analysis derived for equation \ref{sg}, see in particular \cite{bjmr}, we will prove global existence in time for \eqref{w-sg} for the sub-critical case and we will give general blow up criteria for the super-critical and critical case. The sub/super-critical case refers to the sharp constant of the associated Moser-Trudinger inequality, as it will be clear in the sequel.

\medskip

Before stating the results, let us fix some notation. Given $T>0$ and a metric space $X$ we will denote $C([0,T];X)$ by $C_T(X)$. $C^k_T(X)$ and $L^k_T(X)$, $k\geq1$, are defined in an analogous way. When we are taking time derivative for $t\in[0,T]$ we are implicitly taking right (resp. left) derivative at the endpoint $t=0$ (resp. $t=T$). When it will be clear from the context we will simply write $H^1, L^2$ to denote $H^1(M), L^2(M)$, respectively, and
$$
	\|u\|_{H^1(M)}^2=\|\n u\|_{L^2(M)}^2+\|u\|_{L^2(M)}^2.
$$	

\medskip

Our first main result is to show that the initial value problem for \eqref{w-sg} is locally well-posed in $H^1\times L^2$.
\begin{theorem}
\label{th.local-sg}
Let $\rho_1,\rho_2\in\mathbb{R}$. Then, for any $(u_0,u_1)\in H^1(M)\times L^2(M)$ such that $\int_{M}u_1=0$, there exist $T=T(\rho_1, \rho_2, \|u_0\|_{H^1}, \|u_1\|_{L^2})>0$ and a unique, stable solution, i.e. depending continuously on $(u_0,u_1)$,
$$
	u:[0,T]\times M\to\mathbb{R}, \quad u\in C_T(H^1)\cap C_T^1(L^2),
$$
of \eqref{w-sg} with initial data
$$
\left\{	\begin{array}{l}
					u(0,\cdot)=u_0, \\
					\partial_t u(0,\cdot)=u_1.
				\end{array} \right.
$$
Furthermore
\begin{align} \label{u-const}
\int_{M}u(t,\cdot)=\int_{M}u_0 \quad \mbox{for all } t\in[0,T].
\end{align}
\end{theorem}

\medskip

\begin{remark}
The assumption on the initial datum $u_1$ to have zero average guarantees that the average \eqref{u-const} of the solution $u(t,\cdot)$ to \eqref{w-sg} is preserved in time. A consequence of the latter property is that the energy $E(u(t,\cdot))$ given in \eqref{energy-sg} is preserved in time as well, which will be then crucially used in the sequel, see the discussion later on.
\end{remark}

The proof is based on a fixed point argument and the standard Moser-Trudinger inequality \eqref{mt}, see section \ref{sec:proof}. Once the local existence is established we address the existence of a global solution to \eqref{w-sg}. Indeed, by exploiting an energy argument jointly with the Moser-Trudinger inequality associated to \eqref{sg}, see \eqref{mt-sg}, we deduce our second main result.
\begin{theorem}
\label{th.global-sg}
Suppose $\rho_1,\rho_2<8\pi.$ Then, for any $(u_0,u_1)\in H^1(M)\times L^2(M)$ such that $\int_{M}u_1=0$, there exists a unique global solution $u\in C(\R^+;H^1)\cap C^1(\R^+;L^2)$ of \eqref{w-sg} with initial data $(u_0,u_1)$.
\end{theorem}

The latter case $\rho_1,\rho_2<8\pi$ is referred as the sub-critical case in relation to the sharp constant $8\pi$ in the Moser-Trudinger inequality \eqref{mt-sg}. The critical and super-critical case in which $\rho_i\geq8\pi$ for some $i$ is subtler since the solutions to \eqref{sg} might blow up. However, by exploiting the recent analysis concerning \eqref{sg}, see in particular \cite{bjmr}, based on improved versions of the Moser-Trudinger inequality, see Proposition \ref{prop-sg-mt}, we are able to give quite general blow up criteria for \eqref{w-sg}. Our third main result is the following.
\begin{theorem}
\label{th.con-sg}
Suppose $\rho_i\geq8\pi$ for some $i$. Let $(u_0,u_1)\in H^1(M)\times L^2(M)$ be such that $\int_{M}u_1=0$ and let $u$ be the solution of \eqref{w-sg} obtained in Theorem \ref{th.local-sg}. Suppose that $u$ exists in $[0,T_0)$ for some $T_0<+\infty$ and it can not be extended beyond $T_0$. Then, there exists a sequence $t_k\to T_0^{-}$ such that
\begin{align*}
	\lim_{k\to+\infty}\|\nabla u(t_k,\cdot)\|_{L^2}=+\infty, \quad \lim_{k\to+\infty}\max\left\{\int_{M}e^{u(t_k,\cdot)}, \int_{M}e^{-u(t_k,\cdot)}\right\} =+\infty.
\end{align*}
Furthermore, if $\rho_1\in[8m_1\pi,8(m_1+1)\pi)$ and $\rho_2\in[8m_2\pi,8(m_2+1)\pi)$ for some $m_1,m_2\in\mathbb{N}$, then, there exist points $\{x_1,\dots,x_{m_1}\}\subset M$ such that for any $\varepsilon>0,$ either
$$
\lim_{k\to+\infty}\frac{\int_{\bigcup_{l=1}^{m_1}B(x_{l},\varepsilon)}e^{u(t_k,\cdot)}}{\int_{M}e^{u(t_k,\cdot)}}\geq 1-\varepsilon,
$$
or there exist points $\{y_1,\dots,y_{m_2}\}\subset M$ such that for any $\varepsilon>0,$
$$
\lim_{k\to+\infty}\frac{\int_{\bigcup_{l=1}^{m_2}B(y_{l},\varepsilon)}e^{-u(t_k,\cdot)}}{\int_{M}e^{-u(t_k,\cdot)}}\geq 1-\varepsilon.
$$
\end{theorem}

The latter result shows that, once the two parameters $\rho_1, \rho_2$ are fixed in a critical or super-critical regime, the finite time blow up of the solutions to \eqref{w-sg} yields the following alternative: either the measure $e^{u(t_k)}$ (after normalization) concentrates around (at most) $m_1$ points on $M$ (i.e. it resembles a $m_1$-bubble) or $e^{-u(t_k)}$ concentrates around $m_2$ points on $M$. We point out this is new for the mean field equation \eqref{w-mf} as well and generalizes the previous blow up criterion \eqref{mf-cr} obtained in \cite{cy} for $\rho=8\pi$. More precisely, the general blow up criteria for the super-critical mean field equation are the following.
\begin{corollary}
\label{cr1.1}
Suppose $\rho\in[8m\pi,8(m+1)\pi)$ for some $m\in\mathbb{N}$, $m\geq1$. Let $(u_0,u_1)\in H^1(M)\times L^2(M)$ be such that $\int_{M}u_1=0$, and let $u$ be a solution of \eqref{w-mf}, where $\S^2$ is replaced by a compact surface $M$. Suppose that $u$ exists in $[0,T_0)$ for some $T_0<+\infty$ and it can not be extended beyond $T_0$. Then, there exist a sequence $t_k\to T_0^{-}$ and $m$ points $\{p_1,\dots,p_m\}\subset M$ such that for any $\varepsilon>0$,
$$\lim_{k\to\infty}\frac{\int_{\bigcup_{l=1}^{m}B(p_{l},\varepsilon)}e^{u(t_k,\cdot)}}{\int_{M}e^{u(t_k,\cdot)}}
\geq1-\varepsilon.$$
\end{corollary}

Finally, it is worth to point out some possible generalizations of the results so far.
\begin{remark}
We may consider the following more general weighted problem
$$
\partial_t^2u-\Delta_gu=\rho_1\left(\frac{h_1e^u}{\int_{M}h_1e^u}-\frac{1}{|M|}\right)
-\rho_2\left(\frac{h_2e^{-u}}{\int_{M}h_2e^{-u}}-\frac{1}{|M|}\right),
$$
where $h_i=h_i(x)$ are two smooth functions such that $\frac 1C \leq h_i \leq C$ on $M$, $i=1,2$, for some $C>0$. It is easy to check that Theorems \ref{th.local-sg}, \ref{th.global-sg} and \ref{th.con-sg} extend to this case as well. The same argument applies also to the Toda system \eqref{w-toda}.

On the other hand, motivated by several applications in mathematical physics \cite{a, ons, ss} we may consider the following asymmetric sinh-Gordon wave equation
$$
\partial_t^2u-\Delta_gu=\rho_1\left(\frac{e^u}{\int_{M}e^u}-\frac{1}{|M|}\right)
-\rho_2\left(\frac{e^{-au}}{\int_{M}e^{-au}}-\frac{1}{|M|}\right),
$$
with $a>0$. For $a=2$, which corresponds to the Tzitz\'eica equation, we can exploit the detailed analysis in \cite{jy} to derive Theorems \ref{th.local-sg}, \ref{th.global-sg} and \ref{th.con-sg} for this case as well (with suitable modifications accordingly to the associated Moser-Trudinger inequality). On the other hand, for general $a>0$ the complete analysis is still missing and we can rely for example on \cite{j4} to get at least the existence results of Theorems \ref{th.local-sg}, \ref{th.global-sg}.
\end{remark}

\

We next consider the wave equation associated to some general Toda system,
\begin{equation}
\label{w-toda}
\p_t^2u_i-\Delta_gu_i=\sum_{j=1}^na_{ij}\,\rho_j\left(\frac{e^{u_j}}{\int_{M}e^{u_j}}-\frac{1}{|M|}\right) \quad \mbox{on }  M, \quad i=1,\dots,n,
\end{equation}
where $\rho_i$, $i=1,\dots,n$ are real parameters and $A_n=(a_{ij})_{n\times n}$ is the following rank $n$ Cartan matrix for $SU(n+1)$:
\begin{equation} \label{matr}
	{A}_n=\left(\begin{matrix}
2 & -1 & 0 & \cdots & 0 \\
-1 & 2 & -1 & \cdots & 0 \\
\vdots & \vdots &\vdots  & \ddots & \vdots\\
0 & \cdots & -1 & 2 & -1 \\
0 & \cdots &  0 & -1 & 2
\end{matrix}\right).
\end{equation}
The stationary equation related to \eqref{w-toda} is the following Toda system:
\begin{equation}
\label{toda}
-\Delta_gu_i=\sum_{j=1}^na_{ij}\,\rho_j\left(\frac{e^{u_j}}{\int_{M}e^{u_j}}-\frac{1}{|M|}\right), \quad i=1,\dots,n,
\end{equation}
which has been extensively studied since it has several applications both in mathematical physics and in geometry, for example non-abelian Chern-Simons theory \cite{d1, tar2, y} and holomorphic curves in $\mathbb{C}\mathbb{P}^n$ \cite{bw, cal, lwy-c}. There are by now many results concerning Toda-type systems in particular regarding existence of solutions \cite{bjmr, jkm, mr}, blow up analysis \cite{jlw, lwy} and classification issues \cite{lwy-c}.

\medskip

On the other hand, only partial results concerning the wave equation associated to the Toda system \eqref{w-toda} were obtained in \cite{cy} which we recall here. First, the local well-posedness of \eqref{w-toda} analogously as in Theorem \ref{th.local-sg} is derived for a general $n\times n$ symmetric matrix $A_n$. Second, by assuming $A_n$ to be a positive definite symmetric matrix with non-negative entries the authors were able to deduce a global existence result by exploiting a Moser-Trudinger-type inequality suitable for this setting, see \cite{sw}. On the other side, no results are available neither for mixed positive and negative entries of the matrix $A_n$ (which are relevant in mathematical physics and in geometry, see for example the above Toda system) nor for blow up criteria. Our aim is to complete the latter analysis.

\medskip

Before stating the results, let us fix some notation for the system setting. We denote the product space as $(H^1(M))^n=H^1(M)\times\dots\times H^1(M)$. To simplify the notation, to take into account an element $(u_1,\dots,u_n)\in (H^1(M))^n$ we will rather write $H^1(M)\ni\u: M\mapsto(u_1,\dots,u_n)\in\R^n$. With a little abuse of notation we will write $\int_{M} \u$ when we want to consider the integral of each component $u_i$, $i=1,\dots,n$.

\medskip

Since the local well-posedness of \eqref{w-toda} is already known from \cite{cy}, our first result concerns the global existence in time.
\begin{theorem}
\label{th.global-toda}
Suppose $\rho_i<4\pi$ for all $i=1,\dots,n.$ Then, for any $(\u_0,\u_1)\in H^1(M)\times L^2(M)$ such that $\int_{M}\u_1=0$, there exists a unique global solution
$$
	\u:\R^+\times M\to\mathbb{R}^n, \quad \u\in C(\R^+;H^1)\cap C^1(\R^+;L^2),
$$
of \eqref{w-toda} with initial data
$$
\left\{	\begin{array}{l}
					\u(0,\cdot)=\u_0, \\
					\partial_t \u(0,\cdot)=\u_1.
				\end{array} \right.
$$
\end{theorem}

The latter result follows by an energy argument and a Moser-Trudinger-type inequality for systems as obtained in \cite{jw}. On the other hand, when $\rho_i\geq4\pi$ for some $i$ the Moser-Trudinger inequality does not give any control and the solutions of \eqref{toda} might blow up. In the latter case, by exploiting improved versions of the Moser-Trudinger inequality for the system recently derived in \cite{b} we are able to give the following general blow up criteria.
\begin{theorem}
\label{th.blowup-toda}
Suppose $\rho_i\geq4\pi$ for some $i$. Let $(\u_0,\u_1)\in H^1(M)\times L^2(M)$ be such that $\int_{M}\u_1=0$, and let $\u$ be the solution of \eqref{w-toda}. Suppose that $\u$ exists in $[0,T_0)$ for some $T_0<\infty$ and it can not be extended beyond $T_0$. Then, there exists a sequence $t_k\to T_0^{-}$ such that
\begin{align*}
	\lim_{k\to+\infty} \max_j \|\nabla u_j(t_k,\cdot)\|_{L^2}=+\infty, \quad \lim_{k\to+\infty}\max_j\int_{M}e^{u_j(t_k,\cdot)}=+\infty.
\end{align*}
Furthermore, if $\rho_i\in[4m_i\pi,4(m_i+1)\pi)$ for some $m_i\in\mathbb{N}$, $i=1,\dots,n$, then there exists at least one index $j\in\{1,\dots,n\}$ and $m_j$ points $\{x_{j,1},\dots,x_{j,m_j}\}\in M$ such that for any $\varepsilon>0$,
$$\lim_{k\to\infty}\frac{\int_{\bigcup_{l=1}^{m_j}B(x_{j,l},\varepsilon)}e^{u_j(t_k,\cdot)}}{\int_{M}e^{u_j(t_k,\cdot)}}
\geq1-\varepsilon.$$
\end{theorem}

Therefore, for the finite time blow up solutions to \eqref{w-toda} there exists at least one component $u_j$ such that the measure $e^{u(t_k)}$ (after normalization) concentrates around (at most) $m_j$ points on $M$. One can compere this result with the on for the sinh-Gordon equation \eqref{w-sg} or the mean field equation \eqref{w-mf}, see Theorem \ref{th.con-sg} and Corollary \ref{cr1.1}, respectively.

\medskip

Finally, we have the following possible generalization of the system \eqref{w-toda}.
\begin{remark}
We point out that since the improved versions of the Moser-Trudinger inequality in \cite{b} hold for general symmetric, positive definite matrices $A_n=(a_{ij})_{n\times n}$ with non-positive entries outside the diagonal, we can derive similar existence results and blow up criteria as in Theorems \ref{th.global-toda}, \ref{th.blowup-toda}, respectively, for this general class of matrices as well. In particular, after some simple transformations (see for example the introduction in \cite{ajy}) we may treat the following Cartan matrices:
$$
B_n=\left(\begin{matrix}
2 & -1 & 0 & \cdots & 0 \\
-1 & 2 & -1 & \cdots & 0 \\
\vdots & \vdots &\vdots  & \ddots & \vdots\\
0 & \cdots & -1 & 2 & -2 \\
0 & \cdots &  0 & -1 & 2
\end{matrix}\right),
\quad  {C}_n=\left(\begin{matrix}
2 & -1 & 0 & \cdots & 0 \\
-1 & 2 & -1 & \cdots & 0 \\
\vdots & \vdots &\vdots  & \ddots & \vdots\\
0 & \cdots & -1 & 2 & -1 \\
0 & \cdots &  0 & -2 & 2
\end{matrix}\right), \vspace{0.2cm}
$$
$$
{G}_2=\left(\begin{matrix}
2&-1\\
-3&2
\end{matrix}\right),
$$
which are relevant in mathematical physics, see for example \cite{d1}. To simplify the presentation we give the details just for the matrix $A_n$ in \eqref{matr}.
\end{remark}

\medskip

The paper is organized as follows. In section \ref{sec:prelim} we collect some useful results and in section \ref{sec:proof} we prove the main results of this paper: local well-posedness, global existence and blow up criteria.

\medskip

\section{Preliminaries} \label{sec:prelim}
In this section we collect some useful results concerning the stationary sinh-Gordon equation \eqref{sg}, Toda system \eqref{toda} and the solutions of wave equations which will be used in the proof of the main results in the next section.

\medskip

In the sequel the symbol $\ov u$ will denote the average of $u$, that is
$$
\ov u= \fint_{M} u=\frac{1}{|M|}\int_{M} u.
$$
Let us start by recalling the well-known Moser-Trudinger inequality
\begin{equation} \label{mt}
	8\pi \log\int_{M} e^{u-\ov u} \leq \frac 12 \int_{M} |\n u|^2 + C_{(M,g)}\,, \quad u\in H^1(M).
\end{equation}
For the sinh-Gordon equation \eqref{sg}, a similar sharp inequality was obtained in \cite{os},
\begin{equation} \label{mt-sg}
8\pi \left(\log\int_{M} e^{u-\ov u} + \log\int_{M} e^{-u+\ov u}\right) \leq \frac 12 \int_{M} |\n u|^2 + C_{(M,g)}\,,\quad u\in H^1(M).
\end{equation}
We recall now some of main features concerning the variational analysis of the sinh-Gordon equation \eqref{sg} recently derived in \cite{bjmr}, which will be exploited later on. First of all, letting $\rho_1, \rho_2\in\R$ the associated Euler-Lagrange functional for equation \eqref{sg} is given by $J_{\rho_1,\rho_2}:H^1(M)\to \R$,
\begin{equation}
\label{functional-sg}
J_{\rho_1,\rho_2}(u)=\frac12\int_{M}|\nabla u|^2-\rho_1\log\int_{M}e^{u-\overline{u}} -\rho_2\log\int_{M}e^{-u+\overline{u}}.
\end{equation}
Observe that if $\rho_1, \rho_2\leq8\pi$, by \eqref{mt-sg} we readily have
$$
J_{\rho_1,\rho_2}(u)\geq -C,
$$
for any $u\in H^1(M)$, where $C>0$ is a constant independent of $u$. On the other hand, as soon as $\rho_i>8\pi$ for some $i=1,2$ the functional $J_{\rho_1,\rho_2}$ is unbounded from below. To treat the latter super-critical case one needs improved versions of the Moser-Trudinger inequality \eqref{mt-sg} which roughly assert that the more the measures $e^u, e^{-u}$ are spread over the surface the bigger is the constant in the left hand side of \eqref{mt-sg}. More precisely, we have the following result.
\begin{proposition} \emph{(\cite{bjmr})} \label{prop-sg-mt}
Let $\d, \th>0$, $k,l\in\N$ and $\{\o_{1,i},\o_{2,j}\}_{i\in\{1,\dots,k\},j\in\{1,\dots,l\}}\subset M$ be such that
\begin{align*}
	& d(\o_{1,i},\o_{1,i'})\ge\d,\quad \forall \, i,i'\in\{1,\dots,k\}, \, i\ne i', \\
	& d(\o_{2,j},\o_{2,j'})\ge\d,\quad \forall \, j,j' \in\{1,\dots,l\}, \, j\ne j'.
\end{align*}
Then, for any $\e>0$ there exists $C=C\left(\e,\d,\th,k,l,M\right)$ such that if $u\in H^1(M)$ satisfies
\begin{align*}
\int_{\o_{1,i}} e^{u} \ge \th\int_{M} e^{u}, \,\, \forall i\in\{1,\dots,k\},
\qquad \int_{\o_{2,j}} e^{-u} \ge \th\int_{M} e^{-u}, \,\, \forall j\in\{1,\dots,l\},
\end{align*}
it follows that
$$
8k\pi\log\int_{M} e^{u-\ov{u}}+8l\pi\log\int_{M} e^{-u+\ov u}\leq \frac{1+\e}{2}\int_{M} |\n u|^2\,dV_g+C.
$$
\end{proposition}
From the latter result one can deduce that if the $J_{\rho_1,\rho_2}(u)$ is large negative at least one of the two measures $e^u, e^{-u}$ has to concentrate around some points of the surface.
\begin{proposition} \emph{(\cite{bjmr})} \label{prop-sg}
Suppose $\rho_i\in(8m_i\pi,8(m_i+1)\pi)$ for some $m_i\in\N$, $i=1,2$ ($m_i\geq 1$ for some $i=1,2$). Then, for any $\varepsilon, r>0$ there exists $L=L(\varepsilon,r)\gg1$ such that for any $u\in H^1(M)$ with $J_{\rho_1,\rho_2}(u)\leq-L$, there are either some $m_1$ points $\{x_1,\dots,x_{m_1}\}\subset M$ such that
$$
\frac{\int_{\cup_{l=1}^{m_1}B_r(x_l)}e^{u}}{\int_{M}e^u}\geq 1-\varepsilon,
$$
or some $m_2$ points $\{y_1,\dots,y_{m_2}\}\subset M$ such that
$$
\frac{\int_{\cup_{l=1}^{m_2}B_r(y_l)}e^{-u}}{\int_{M}e^{-u}}\geq 1-\varepsilon.
$$
\end{proposition}
We next briefly recall some variational aspects of the stationary Toda system \eqref{toda}. Recall the matrix $A_n$ in \eqref{matr} and the notation of $\u$ introduced before Theorem~\ref{th.global-toda} and write $\rho=(\rho_1,\dots,\rho_n)$. The associated functional for the system \eqref{toda} is given by $J_{\rho}:H^1(M)\to \R$,
\begin{equation}
\label{functional-toda}
J_{\rho}(\u)=\frac12\int_{M}\sum_{i,j=1}^n a^{ij}\langle\nabla u_i,\nabla u_j\rangle
-\sum_{i=1}^n\rho_i\log\int_{M}e^{u_i-\overline{u}_i},
\end{equation}
where $(a^{ij})_{n\times n}$ is the inverse matrix $A_n^{-1}$ of $A_n$. A Moser-Trudinger inequality for \eqref{functional-toda} was obtained in \cite{jw}, which asserts that
\begin{equation} \label{mt-toda}
J_{\rho}(\u)\geq C,
\end{equation}
for any $\u\in H^1(M)$, where $C$ is a constant independent of $\u$, if and only if $\rho_i\leq 4\pi$ for any $i=1,\dots,n.$ In particular, if $\rho_i>4\pi$ for some $i=1,\dots,n$ the functional $J_\rho$ is unbounded from below. As for the sinh-Gordon equation \eqref{sg} we have improved versions of the Moser-Trudinger inequality \eqref{mt-toda} recently derived in \cite{b} (see also \cite{bjmr}) which yield concentration of the measures $e^{u_j}$ whenever $J_{\rho}(\u)$ is large negative.
\begin{proposition} \emph{(\cite{b, bjmr})} \label{prop-toda}
Suppose $\rho_i\in(4m_i\pi,4(m_i+1)\pi)$ for some $m_i\in\N,~i=1,\dots,n$ ($m_i\geq1$ for some $i=1,\dots,n$). Then, for any $\varepsilon, r>0$ there exists $L=L(\varepsilon,r)\gg 1$ such that for any $\u\in H^1(M)$ with $J_{\rho}(\u)\leq-L$, there exists at least one index $j\in\{1,\dots,n\}$ and $m_j$ points $\{x_1,\dots,x_{m_j}\}\subset M$ such that
$$
\frac{\int_{\cup_{l=1}^{m_j}B_r(x_l)}e^{u_j}}{\int_{M}e^{u_j}}\geq 1-\varepsilon.
$$
\end{proposition}

\medskip

Finally, let us state a standard result concerning the wave equation, that is the Duhamel principle. Let us first recall that every function in $L^2(M)$ can be decomposed as a convergent sum of eigenfunctions of the Laplacian $\Delta_g$ on $M$. Then, one can define the operators $\cos(\sqrt{-\Delta_g})$ and $\frac{\sin(\sqrt{-\Delta_g})}{\sqrt{-\Delta_g}}$ acting on $L^2(M)$ using the spectral theory. Consider now the initial value problem
\begin{equation}
\label{wave-eq}
\left\{\begin{array}{l}
\partial_t^2v-\Delta_gv=f(t,x),\vspace{0.2cm}\\
v(0,\cdot)=u_0,\quad \partial_tv(0,\cdot)=u_1,
\end{array}\right.
\end{equation}
on $[0,+\infty)\times M$. Recall the notation of $C_T(X)$ and $\u$ before Theorems \ref{th.local-sg} and \ref{th.global-toda}, respectively. Then, the following Duhamel formula holds true.
\begin{proposition}
\label{pr2.du}
Let $T>0$, $(u_0,u_1)\in H^1(M)\times L^2(M)$ and let $f\in L^1_T(L^2(M))$. Then, \eqref{wave-eq} has a unique solution
$$
	v:[0,T)\times M\to\mathbb{R}, \quad v\in C_T(H^1)\cap C_T^1(L^2),
$$
given by
\begin{align}
\label{2.wave-ex}
v(t,x)=\cos\left(t\sqrt{-\Delta_g}\right)u_0+\frac{\sin(t\sqrt{-\Delta_g})}{\sqrt{-\Delta_g}}\,u_1+
\int_0^t\frac{\sin\bigr((t-s)\sqrt{-\Delta_g}\bigr)}{\sqrt{-\Delta_g}}\,f(s)\,\mathrm{d}s.
\end{align}
Furthermore, it holds
\begin{equation}
\|v\|_{C_T(H^1)}+\|\partial_tv\|_{C_T(L^2)}\leq 2\left(\|u_0\|_{H^1}+\|u_1\|_{L^2}+\|f\|_{L_T^1(L^2)}\right).
\end{equation}
The same results hold as well if $u_0,u_1,$ and $f(t,\cdot)$ are replaced by $\u_0, \u_1$ and $\mathbf{f}(t,\cdot)$, respectively.
\end{proposition}

\medskip

\section{Proof of the main results} \label{sec:proof}
In this section we derive the main results of the paper, that is local well-posedness, global existence and blow up criteria for the wave sinh-Gordon equation \eqref{w-sg}, see Theorems \ref{th.local-sg}, \ref{th.global-sg} and \ref{th.con-sg}, respectively. Since the proofs of global existence and blow up criteria for the wave Toda system \eqref{w-toda} (Theorems \ref{th.global-toda} and \ref{th.blowup-toda}) are obtained by similar arguments, we will present full details for what concerns the  wave sinh-Gordon equation and point out the differences in the two arguments, where necessary.

\subsection{Local and global existence.}
We start by proving the local well-posedness of the wave sinh-Gordon equation \eqref{w-sg}. The proof is mainly based on a fixed point argument and the Moser-Trudinger inequality \eqref{mt}.

\medskip

\noindent {\em Proof of Theorem \ref{th.local-sg}.} Let $(u_0,u_1)\in H^1(M)\times L^2(M)$ be such that $\int_{M}u_1=0$. Take $T>0$ to be fixed later on. We set
\begin{align} \label{R}
R=3\left(\|u_0\|_{H^1}+\|u_1\|_{L^2}\right),\quad I=\fint_{ M}u_0=\frac{1}{|M|}\int_{M}u_0,
\end{align}
and we introduce the space $B_T$ given by
\begin{align*}
B_T=\left\{u\in C_T(H^1(M))\cap C_T^1(L^2(M)) \,:\, \|u\|_*\leq R,~ \fint_{M}u(s,\cdot)=I \, \mbox{ for all } s\in[0,T]\right\},
\end{align*}
where
\begin{align*}
\|u\|_*=\|u\|_{C_T(H^1)}+\|\partial_tu\|_{C_T(L^2)}.
\end{align*}
For $u\in B_T$ we consider the initial value problem
\begin{equation} \label{w-Bt}
\left\{\begin{array}{l}
\partial_t^2v-\Delta_gv=f(s,x)=\rho_1\left(\frac{e^u}{\int_{M}e^u}-\frac{1}{|M|}\right)
-\rho_2\left(\frac{e^{-u}}{\int_{M}e^{-u}}-\frac{1}{|M|}\right), \vspace{0.2cm}\\
v(0,\cdot)=u_0,\quad \partial_tv(0,\cdot)=u_1,
\end{array}\right.
\end{equation}
on $[0,T]\times M$. Applying Proposition \ref{pr2.du} we deduce the existence of a unique solution of \eqref{w-Bt}.

\medskip

\noindent
\textbf{Step 1.} We aim to show that $v\in B_{T}$ if $T$ is taken sufficiently small. Indeed, still by Proposition \ref{pr2.du} we have
\begin{equation}
\label{3.ine}
\begin{aligned}
\|v\|_*\leq~&2\left(\|u_0\|_{H^1}+\|u_1\|_{L^2}\right)
+2\rho_1\int_0^T\left\|\left(\frac{e^u}{\int_{M}e^u}-\frac{1}{|M|}\right)\right\|_{L^2}\mathrm{d}s\\
&+2\rho_2\int_0^T\left\|\left(\frac{e^{-u}}{\int_{M}e^{-u}}-\frac{1}{|M|}\right)\right\|_{L^2}\mathrm{d}s.
\end{aligned}
\end{equation}
Since $u\in B_T$ we have $\fint_{M}u(s,\cdot)=I$ for all $s\in[0,T]$ and therefore, by the Jensen inequality,
\begin{align*}
\fint_{M}e^{u}\geq e^{\fint_{M}u}=e^{I}\quad\mbox{and}\quad
\fint_{M}e^{-u}\geq e^{-\fint_{M}u}=e^{-I}.
\end{align*}
Therefore, we can bound the last two terms on the right hand side of \eqref{3.ine} by
\begin{align*}
CT(|\rho_1|+|\rho_2|)+CT|\rho_1|e^{-I}\max_{s\in[0,T]}\|e^{u(s,\cdot)}\|_{L^2}+CT|\rho_2|e^{I}\max_{s\in[0,T]}\|e^{-u(s,\cdot)}\|_{L^2},
\end{align*}
for some $C>0$. On the other hand, recalling the Moser-Trudinger inequality \eqref{mt}, we have for $s\in[0,T]$
\begin{equation}
\label{exp}
\begin{aligned}
\|e^{u(s,\cdot)}\|_{L^2}^2=~&\int_{M}e^{2u(s,\cdot)}=\int_{M}e^{2(u(s,\cdot)-\overline{u})}e^{2I} \\
\leq~&C\exp\left(\frac{1}{4\pi}\int_{M}|\nabla u(s,\cdot)|^2\right)e^{2I} \leq Ce^{2I}e^{\frac{1}{4\pi}R^2},
\end{aligned}
\end{equation}
for some $C>0$, where we used $\|u\|_*\leq R$. Similarly, we have
$$
	\|e^{-u(s,\cdot)}\|_{L^2}^2\leq Ce^{-2I}e^{\frac{1}{4\pi}R^2}.
$$
Hence, recalling the definition of $R$ in \eqref{R}, by \eqref{3.ine} and the above estimates we conclude
\begin{align*}
\|v\|_* &\leq 2\left(\|u_0\|_{H^1}+\|u_1\|_{L^2}\right)+CT(|\rho_1|+|\rho_2|)+CT(|\rho_1|+|\rho_2|)e^{\frac{1}{8\pi}R^2} \\
	& = \frac23 R+CT(|\rho_1|+|\rho_2|)+CT(|\rho_1|+|\rho_2|)e^{\frac{1}{8\pi}R^2}.
\end{align*}
Therefore, If $T>0$ is taken sufficiently small, $T=T(\rho_1,\rho_2,\|u_0\|_{H^1},\|u_1\|_{L^2})$, then $\|v\|_*\leq R$.

\medskip

Moreover, observe that if we integrate both sides of \eqref{w-Bt} on $M$ we get
$$
	\partial_t^2\ov{v}(s)=0, \quad \mbox{for all } s\in[0,T]
$$
and hence,
$$
	\partial_t\ov{v}(s)=\partial_t\ov{v}(0)=\ov{u}_1=0, \quad \mbox{for all } s\in[0,T].
$$
It follows that
$$
\fint_{M}v(s,\cdot)=\fint_{M}v(0,\cdot)=\fint_{M}u_0=I \quad \mbox{for all } s\in[0,T].
$$
Thus, for this choice of $T$ we conclude that $v\in B_T$.

\medskip

Therefore, we can define a map
$$
	\mathcal{F}:B_T\to B_T, \quad v=\mathcal{F}(u).
$$

\noindent
\textbf{Step 2.} We next prove that by taking a smaller $T$ if necessary, $\mathcal{F}$ is a contraction. Indeed, let $u_1,u_2\in B_T$ be such that $v_i=\mathcal{F}(u_i), \,i=1,2$. Then, $v=v_1-v_1$ satisfies
\begin{equation*}
\begin{cases}
\p_t^2v-\Delta_gv=\rho_1\left(\frac{e^{u_1}}{\int_{M}e^{u_1}}-\frac{e^{u_2}}{\int_{M}e^{u_2}}\right)
-\rho_2\left(\frac{e^{-u_1}}{\int_{M}e^{-u_1}}-\frac{e^{-u_2}}{\int_{M}e^{-u_2}}\right),\\
v(0,\cdot)=0,\quad \p_tv(0,\cdot)=0.
\end{cases}
\end{equation*}
Hence, by Proposition \ref{pr2.du} we have
\begin{equation}
\label{contraction-estimate}
\begin{aligned}
\|v\|_*\leq~&2\rho_1\int_0^T\left\|\left(\frac{e^{u_1}}{\int_{ M}e^{u_1}}
-\frac{e^{u_2}}{\int_{ M}e^{u_2}}\right)\right\|_{L^2}\mathrm{d}s\\
&+2\rho_2\int_0^T\left\|\left(\frac{e^{-u_1}}{\int_{ M}e^{-u_1}}
-\frac{e^{-u_2}}{\int_{ M}e^{-u_2}}\right)\right\|_{L^2}\mathrm{d}s.
\end{aligned}
\end{equation}
For $s\in[0,T]$, we use the following decomposition,
\begin{equation} \label{estimate-1}
\left\|\left(\frac{e^{u_1(s,\cdot)}}{\int_{ M}e^{u_1(s,\cdot)}}
-\frac{e^{u_2(s,\cdot)}}{\int_{ M}e^{u_2(s,\cdot)}}\right)\right\|_{L^2}
	 \leq \left\|\frac{e^{u_1}-e^{u_2}}{\int_{ M}e^{u_1}}\right\|_{L^2}  +\left\|\frac{e^{u_2}\left(\int_{ M}e^{u_1}-\int_{ M}e^{u_2}\right)} {(\int_{ M}e^{u_1})(\int_{ M}e^{u_2})}\right\|_{L^2}.
\end{equation}
Reasoning as before, the first term in the righ hand side of the latter estimate is bounded by
\begin{align}
&Ce^{-I}\left\|(u_1(s,\cdot)-u_2(s,\cdot))(e^{_1(s,\cdot)}+e^{u_2(s,\cdot)})\right\|_{L^2} \nonumber\\
&\leq C e^{-I}\|u_1(s,\cdot)-u_2(s,\cdot)\|_{L^4}\left(\|e^{u_1(s,\cdot)}\|_{L^4}+\|e^{u_2(s,\cdot)}\|_{L^4}\right), \label{eq.rhs1}
\end{align}
for some $C>0$, where we used the H\"older inequality. Moreover, we have
\begin{equation*}
\begin{aligned}
\|e^{u_i(s,\cdot)}\|_{L^4}^4=~&\int_{ M}e^{4(u_i(s,\cdot)-\overline{u}_1(s))}e^{4I} \leq Ce^{4I}\exp\left(\frac{1}{\pi}\int_{ M}|\nabla u_i(s,\cdot)|^2\right)\\
\leq~& Ce^{4I}e^{\frac{1}{\pi}R^2}, \quad i=1,2,
\end{aligned}
\end{equation*}
for some $C>0$. Using the latter estimate for the second term in \eqref{eq.rhs1} and the Sobolev inequality for the first term, we can bound \eqref{eq.rhs1} by
\begin{align*}
Ce^{\frac{1}{4\pi}R^2}\|u_1-u_2\|_{H^1}
\end{align*}
and hence
\begin{equation} \label{term1}
	\left\|\frac{e^{u_1}-e^{u_2}}{\int_{ M}e^{u_1}}\right\|_{L^2} \leq Ce^{\frac{1}{4\pi}R^2}\|u_1-u_2\|_{H^1}.
\end{equation}
On the other hand, by using \eqref{exp}, the second term in \eqref{estimate-1} is bounded by
\begin{equation*}
\begin{aligned}
&Ce^{-2 I}\|e^{u_2(s,\cdot)}\|_{L^2}\int_{ M}|u_1(s,\cdot)-u_2(s,\cdot)|\left(e^{u_1(s,\cdot)}+e^{u_2(s,\cdot)}\right)\\
&\leq Ce^{-2I}\|e^{u_2(s,\cdot)}\|_{L^2}
\left(\|e^{u_1(s,\cdot)}\|_{L^2}+\|e^{u_2(s,\cdot)}\|_{L^2}\right)\|u_1(s,\cdot)-u_2(s,\cdot)\|_{L^2}\\
&\leq Ce^{\frac{1}{4\pi}R^2}\|u_1(s,\cdot)-u_2(s,\cdot)\|_{H^1},
\end{aligned}
\end{equation*}
for some $C>0$, where in the last step we used the Sobolev inequality.

\medskip

In conclusion, we have
\begin{align*}
\left\|\left(\frac{e^{u_1(s,\cdot)}}{\int_{ M}e^{u_1(s,\cdot)}}
-\frac{e^{u_2(s,\cdot)}}{\int_{ M}e^{u_2(s,\cdot)}}\right)\right\|_{L^2}
\leq Ce^{\frac{1}{4\pi}R^2}\|u_1(s,\cdot)-u_2(s,\cdot)\|_{L^2}.
\end{align*}
Similarly,
\begin{align*}
\left\|\left(\frac{e^{-u_1(s,\cdot)}}{\int_{ M}e^{-u_1(s,\cdot)}}
-\frac{e^{-u_2(s,\cdot)}}{\int_{ M}e^{-u_2(s,\cdot)}}\right)\right\|_{L^2}
\leq Ce^{\frac{1}{4\pi}R^2}\|u_1(s,\cdot)-u_2(s,\cdot)\|_{L^2}.
\end{align*}
Finally, by the latter estimate, \eqref{term1} and by \eqref{estimate-1}, \eqref{contraction-estimate}, we conclude that
\begin{align*}
\|v\|_* & \leq CT(|\rho_1|+|\rho_2|)e^{\frac{1}{4\pi}R^2}\|u_1(s,\cdot)-u_2(s,\cdot)\|_{H^1} \\
	& \leq CT(|\rho_1|+|\rho_2|)e^{\frac{1}{4\pi}R^2}\|u_1-u_2\|_*\,.
\end{align*}
Therefore, If $T>0$ is taken sufficiently small, $T=T(\rho_1,\rho_2,\|u_0\|_{H^1},\|u_1\|_{L^2})$, then $\mathcal{F}$ a contraction map. The latter fact yields the existence of a unique fixed point for $\mathcal{F}$, which solves \eqref{w-sg} with initial conditions $(u_0, u_1)$.

\medskip

The same arguments with suitable adaptations show that the initial value problem \eqref{w-sg} is locally well-posed so we omit the details. The proof is completed. \hfill $\square$

\

We next prove that if the two parameters in \eqref{w-sg} are taken in a sub-critical regime, then there exists a global solution to the initial value problem associated to \eqref{w-sg}. To this end we will exploit an energy argument jointly with the Moser-Trudinger inequality related to \eqref{sg}, see \eqref{mt-sg}. For a solution $u(t,x)$ to \eqref{w-sg} we define its energy as
\begin{equation}
\label{energy-sg}
E(u(t,\cdot))=\frac12\int_{ M}(|\partial_tu|^2+|\nabla u|^2)-\rho_1\log\int_{ M}e^{u-\overline{u}}
-\rho_2\log\int_{ M}e^{-u+\overline{u}},
\end{equation}
for $t\in[0,T]$. We point out that
$$
	E(u(t,\cdot))=\frac12\int_{ M}|\partial_tu|^2 +J_{\rho_1,\rho_2}(u(t,\cdot)),
$$
where $J_{\rho_1,\rho_2}$ is the functional introduced in \eqref{functional-sg}. We first show that the latter energy is conserved in time along the solution $u$.
\begin{lemma} \label{lem-const}
Let $\rho_1,\rho_2\in\mathbb{R}$ and let $(u_0,u_1)\in H^1( M)\times L^2( M)$ be such that $\int_{ M}u_1=0$. Let $ u\in C_T(H^1)\cap C_T^1(L^2)$, for some $T>0$, be a solution to \eqref{w-sg} with initial data $(u_0, u_1)$ and let $E(u)$ be defined in \eqref{energy-sg}. Then, it holds
$$
	E(u(t,\cdot))=E(u(0,\cdot))	 \quad \mbox{for all } t\in[0,T].
$$
\end{lemma}

\begin{proof}
We will show that 
$$
	\partial_tE(u(t,\cdot))=0 \quad \mbox{for all } t\in[0,T].
$$
We have
\begin{equation}
\label{diff-energy}
\begin{aligned}
\partial_tE(u(t,\cdot))=\int_{ M}(\p_tu)(\p_t^2u)+\int_{ M}\langle\nabla\p_tu,\nabla u\rangle-\rho_1\frac{\int_{ M}e^{u}\p_tu}{\int_{ M}e^{u}}+\rho_2\frac{\int_{ M}e^{-u}\p_tu}{\int_{ M}e^{-u}}.
\end{aligned}
\end{equation}
After integration by parts, the first two terms in the right hand side of the latter equation give
\begin{equation*}
\begin{aligned}
\int_{ M}(\p_tu)(\p_t^2u-\Delta_gu)=\int_{ M}\p_tu\left(\rho_1\left(\frac{e^u}{\int_{ M}e^u}-\frac{1}{|M|}\right)
-\rho_2\left(\frac{e^{-u}}{\int_{ M}e^{-u}}-\frac{1}{|M|}\right)\right),
\end{aligned}
\end{equation*}
where we have used the fact that $u$ satisfies \eqref{w-sg}. Plugging the latter equation into \eqref{diff-energy} we readily have
$$
\partial_tE(u(t),\cdot)=\frac{\rho_2-\rho_1}{|M|}\int_{ M}\p_tu=\frac{\rho_2-\rho_1}{|M|}\,\p_t\left(\int_{ M}u\right)=0 \quad \mbox{for all } t\in[0,T],
$$
since $\int_{ M}u(t,\cdot)=\int_{ M}u_0$ for all $t\in[0,T]$, see Theorem \ref{th.local-sg}. This concludes the proof.
\end{proof}

\

We can now prove the global existence result for \eqref{w-sg} in the sub-critical regime $\rho_1,\rho_2<8\pi$.

\medskip

\noindent {\em Proof of Theorem \ref{th.global-sg}.} Suppose $\rho_1, \rho_2<8\pi$. Let $(u_0,u_1)\in H^1( M)\times L^2( M)$ be such that $\int_{ M}u_1=0$ and let $u$ be the solution to \eqref{w-sg} with initial data $(u_0,u_1)$ obtained in Theorem \ref{th.local-sg}. Suppose that $u$ exists in $[0,T_0)$. With a little abuse of notation $C([0,T_0);H^1)$ will be denoted here still by $C_{T_0}(H^1)$. Analogously we will use the notation $C_{T_0}^1(L^2)$. We have that $u\in C_{T_0}(H^1)\cap C_{T_0}^1(L^2)$ satisfy
\begin{equation*}
\p_t^2u-\Delta u=\rho_1\left(\frac{e^u}{\int_{ M}e^u}-\frac{1}{|M|}\right)-\rho_2\left(\frac{e^{-u}}{\int_{ M}e^{-u}}-\frac{1}{|M|}\right)
\ \ \mbox{on}~[0,T_0)\times M.
\end{equation*}
We claim that
\begin{equation}
\label{a-priori}
\|u\|_{C_{T_0}(H^1)}+\|\partial_t u\|_{C_{T_0}(L^2)}\leq C,
\end{equation}
for some $C>0$ depending only on $\rho_1,\rho_2$ and $(u_0,u_1)$. Once the claim is proven we can extend the solution $u$ for a fixed amount of time starting at any $t\in[0,T_0)$, which in particular implies that the solution $u$ can be extended beyond time $T_0$. Repeating the argument we can extend $u$ for any time and obtain a global solution as desired.

\medskip

Now we shall prove \eqref{a-priori}. We start by recalling that the energy $E(u(t,\cdot))$ in \eqref{energy-sg} is conserved in time, see Lemma \ref{lem-const}, that is,
\begin{equation} \label{const}
E(u(t,\cdot))=E(u(0,\cdot)) \quad \mbox{for all } t\in[0,T_0).
\end{equation}
Suppose first $\rho_1,\rho_2\in(0,8\pi)$. By the Moser-Trudinger inequality \eqref{mt-sg} we have
\begin{equation*}
8\pi\left(\log\int_{ M}e^{u(t,\cdot)-\overline{u}(t)}+\log\int_{ M}e^{-u(t,\cdot)+\overline{u}(t)}\right)
\leq\frac{1}{2}\int_{ M}|\nabla u(t,\cdot)|^2+C, \quad  t\in[0,T_0),
\end{equation*}
where $C>0$ is independent of $u(t,\cdot)$. Observe moreover that by the Jensen inequality it holds
\begin{equation} \label{jensen}
	\log\int_{ M}e^{u(t,\cdot)-\overline{u}(t)}\geq0, \quad \log\int_{ M}e^{-u(t,\cdot)+\overline{u}(t)}\geq0 , \quad  t\in[0,T_0).
\end{equation}
Therefore, letting $\rho=\max\{\rho_1,\rho_2\}$ we have
\begin{align}
	E(u(t,\cdot))&\geq \frac12\int_{ M}(|\partial_tu(t,\cdot)|^2+|\nabla u(t,\cdot)|^2) \nonumber\\
							 & \quad -\rho\left(\log\int_{ M}e^{u(t,\cdot)-\overline{u}(t)}
-\log\int_{ M}e^{-u(t,\cdot)+\overline{u}(t)}\right) \nonumber\\
	 & \geq \frac12\int_{ M}(|\partial_tu(t,\cdot)|^2+|\nabla u(t,\cdot)|^2)-\frac{\rho}{16\pi}\int_{ M} |\nabla u(t,\cdot)|^2 - C\rho, \label{es1}
\end{align}
for $t\in[0,T_0)$, where $C>0$ is independent of $u(t,\cdot)$. Finally, since $\rho<8\pi$ and by using \eqref{const} we deduce
\begin{equation*}
\begin{aligned}
 &\frac12\left(1-\frac{\rho}{8\pi}\right)\left( \|\p_tu(t,\cdot)\|_{L^2}^2+\|\nabla u(t,\cdot)\|_{L^2}^2\right) \\
 &\quad \leq \frac12\int_{ M}\left(|\partial_tu(t,\cdot)|^2+\left(1-\frac{\rho}{8\pi}\right)|\nabla u(t,\cdot)|^2\right) \\
 &\quad \leq E(u(t,\cdot))+C\rho=E(u(0,\cdot))+C\rho,
\end{aligned}
\end{equation*}
where $C>0$ is independent of $u(t,\cdot)$.

On the other hand, to estimate $\|u(t,\cdot)\|_{L^2}$ we recall that $\int_{ M}u(t,\cdot)=\int_{ M}u_0$ for all $t\in[0,T_0)$, see Theorem \ref{th.local-sg}, and use the Poincar\'e inequality to get
\begin{align*}
	\|u(t,\cdot)\|_{L^2} & \leq \|u(t,\cdot)-\ov{u}(t)\|_{L^2} + \|\ov{u}(t)\|_{L^2} \leq C\|\nabla u(t,\cdot)\|_{L^2} + C\ov{u}(t) \\
		& = C\|\nabla u(t,\cdot)\|_{L^2} + C\ov{u}_0,
\end{align*}
where $C>0$ is independent of $u(t,\cdot)$. By the latter estimate and \eqref{es1} we readily have \eqref{a-priori}.

\medskip

Suppose now one of $\rho_1,\rho_2$'s is not positive. Suppose without loss of generality $\rho_1\leq 0$. Then, recalling \eqref{jensen} and by using the standard Moser-Trudinger inequality \eqref{mt} we have
\begin{equation*}
\begin{aligned}
E(u(t,\cdot))\geq~&\frac12\int_{ M}(|\partial_tu(t,\cdot)|^2+|\nabla u(t,\cdot)|^2)-\rho_2\log\int_{ M}e^{u(t,\cdot)-\overline{u}(t)}\\
\geq~&\frac12\int_{ M}\left(|\partial_tu(t,\cdot)|^2+\left(1-\frac{\rho_2}{8\pi}\right)|\nabla u(t,\cdot)|^2\right)-C\rho_2.
\end{aligned}
\end{equation*}
Reasoning as before one can get \eqref{a-priori}.

\medskip

Finally, suppose $\rho_1,\rho_2\leq0$. In this case, we readily have
$$E(u(0,\cdot))=E(u(t,\cdot))\geq\frac12\int_{ M}(|\partial_tu(t,\cdot)|^2+|\nabla u(t,\cdot)|^2),$$
which yields \eqref{a-priori}. The proof is completed. \hfill $\square$

\medskip

\begin{remark} \label{rem-toda}
For what concerns the wave equation associated to the Toda system \eqref{w-toda} we can carry out a similar argument to deduce the global existence result in Theorem~\ref{th.global-toda}. Indeed, for a solution $\u=(u_1,\dots,u_n)$ to \eqref{w-toda} we define its energy as
\begin{equation*}
E(\u(t,\cdot))= \frac12\int_{M}\sum_{i,j=1}^n a^{ij}\left( (\partial_tu_i)(\partial_tu_j) + \langle\nabla u_i,\nabla u_j\rangle \right)
-\sum_{i=1}^n\rho_i\log\int_{ M}e^{u_i-\overline{u}_i},
\end{equation*}
where $(a^{ij})_{n\times n}$ is the inverse matrix $A_n^{-1}$ of $A_n$. Analogous computations as in Lemma~\ref{lem-const} show that the latter energy is conserved in time, i.e.
$$
	E(\u(t,\cdot))=E(\u(0,\cdot)) \quad \mbox{for all } t\in[0,T].
$$
To prove the global existence in Theorem~\ref{th.global-toda} for $\rho_i<4\pi$, $i=1,\dots,n$, one can then follow the argument of Theorem \ref{th.global-sg} jointly with the Moser-Trudinger inequality associated to the Toda system \eqref{toda}, see \eqref{mt-toda}.
\end{remark}

\bigskip

\subsection{Blow up criteria.} We next consider the critical/super-critical case in which $\rho_i\geq8\pi$ for some $i$. The fact that the solutions to \eqref{sg} might blow up makes the problem more delicate. By exploiting the analysis introduced in \cite{bjmr}, in particular the improved version of the Moser-Trudinger inequality in Proposition \ref{prop-sg-mt} and the concentration property in Proposition \ref{prop-sg}, we derive the following general blow up criteria for \eqref{w-sg}. We stress this is new for the wave mean field equation \eqref{w-mf} as well.

\medskip

\noindent {\em Proof of Theorem \ref{th.con-sg}.} Suppose $\rho_i\geq8\pi$ for some $i$. Let $(u_0,u_1)\in H^1( M)\times L^2( M)$ be such that $\int_{ M}u_1=0$ and let $u$ be the solution of \eqref{w-sg} obtained in Theorem \ref{th.local-sg}. Suppose that $u$ exists in $[0,T_0)$ for some $T_0<+\infty$ and it can not be extended beyond $T_0$. Then, we claim that there exists a sequence $t_k\to T_0^-$ such that either
\begin{equation}
\label{one-infty}
\lim_{k\to\infty}\int_{ M}e^{u(t_k,\cdot)}=+\infty \quad \mbox{or} \quad \lim_{k\to\infty}\int_{ M}e^{-u(t_k,\cdot)}=+\infty.
\end{equation}
Indeed, suppose this is not the case. Recall the definition of $E(u)$ in \eqref{energy-sg} and the fact that it is conserved in time \eqref{const}. Recall moreover that $\int_{ M}u(t,\cdot)=\int_{ M}u_0$ for all $t\in[0,T_0)$, see Theorem \ref{th.local-sg}. Then, we would have
\begin{align*}
	& \frac12\int_{ M}(|\partial_tu(t,\cdot)|^2+|\nabla u(t,\cdot)|^2) \\
	& \quad = E(u(t,\cdot)) +\rho_1\log\int_{ M}e^{u(t,\cdot)-\overline{u}(t)} +\rho_2\log\int_{ M}e^{-u(t,\cdot)+\overline{u}(t)} \\
	& \quad \leq E(u(t,\cdot)) +(\rho_2-\rho_1)\overline{u}(t) +C \\
	&	\quad = E(u(0,\cdot)) +(\rho_2-\rho_1)\overline{u}(0) +C \\
	& \quad \leq C \quad \mbox{for all } t\in[0,T_0),
\end{align*}
for some $C>0$ depending only on $\rho_1,\rho_2$ and $(u_0,u_1)$. Thus, we can extend the solution $u$ beyond time $T_0$ contradicting the maximality of $T_0$. We conclude \eqref{one-infty} holds true. Now, since $\overline{u}(t)$ is constant in time The Moser-Trudinger inequality \eqref{mt} yields
\begin{equation}
\lim_{k\to\infty}\|\nabla u(t_k,\cdot)\|_{L^2}=+\infty.
\end{equation}
This concludes the first part of Theorem \ref{th.con-sg}.

\medskip

Finally, suppose $\rho_1\in[8m_1\pi,8(m_1+1)\pi)$ and $\rho_2\in[8m_2\pi,8(m_2+1)\pi)$ for some $m_1,m_2\in\mathbb{N}$, and let $t_k$ be the above defined sequence. Next we take $\tilde{\rho}_i>\rho_i$ such that $\tilde{\rho}_i\in(8m_i\pi,8(m_i+1)\pi),~i=1,2$, and consider the following functional as in \eqref{functional-sg},
\begin{align*}
J_{\tilde{\rho}_1,\tilde{\rho}_2}(u)=\frac12\int_{ M}|\nabla u|^2-\tilde{\rho}_1\int_{ M}e^{u-\overline{u}}-\tilde{\rho}_2\int_{ M}e^{-u+\overline{u}}.
\end{align*}
Since $\tilde{\rho}_i>\rho_i,~i=1,2$ and since $E(u(t_i,\cdot))$, $\overline{u}(t)$ are preserved in time, we have
\begin{equation*}
\begin{aligned}
J_{\tilde{\rho}_1,\tilde{\rho}_2}(u(t_k,\cdot)) &= E(u(t_k,\cdot))-\frac12\int_{ M}|\partial_tu(t,\cdot)|^2 \\
&\quad -(\tilde\rho_1-\rho_1)\log\int_{ M}e^{u(t_k,\cdot)-\overline{u}(t_k,\cdot)}-(\tilde\rho_2-\rho_2)\log\int_{ M}e^{-u(t_k,\cdot)+\overline{u}(t_k,\cdot)} \\
& \leq E(u(0,\cdot))+(\tilde\rho_1-\rho_1)\ov u(0)-(\tilde\rho_2-\rho_2)\ov u(0) \\
& \quad -(\tilde\rho_1-\rho_1)\log\int_{ M}e^{u(t_k,\cdot)}-(\tilde\rho_2-\rho_2)\log\int_{ M}e^{-u(t_k,\cdot)}\to-\infty,
\end{aligned}
\end{equation*}
for $k\to+\infty$, where we used \eqref{one-infty}. Then, by the concentration property in Proposition \ref{prop-sg} applied to the functional $J_{\tilde{\rho}_1,\tilde{\rho}_2}$, for any $\varepsilon>0$ we can find either some $m_1$ points $\{x_1,\dots,x_{m_1}\}\subset M$ such that, up to a subsequence,
$$
\lim_{k\to+\infty}\frac{\int_{\cup_{l=1}^{m_1}B_r(x_l)}e^{u(t_k,\cdot)}}{\int_{ M}e^{u(t_k,\cdot)}}\geq 1-\varepsilon,
$$
or some $m_2$ points $\{y_1,\dots,y_{m_2}\}\subset M$ such that
$$
\lim_{k\to+\infty}\frac{\int_{\cup_{l=1}^{m_2}B_r(y_l)}e^{-u(t_k,\cdot)}}{\int_{ M}e^{-u(t_k,\cdot)}}\geq 1-\varepsilon.
$$
This finishes the last part of Theorem \ref{th.con-sg}. \hfill $\square$

\medskip

\begin{remark}
The general blow up criteria in Theorem \ref{th.blowup-toda} for the wave equation associated to the Toda system \eqref{w-toda} in the critical/super critical regime $\rho_i\geq 4\pi$ are obtained similarly. More precisely, one has to exploit the conservation of the energy of solutions to \eqref{w-toda}, see Remark \ref{rem-toda}, and the concentration property for the Toda system \eqref{toda} in Proposition \ref{prop-toda}.
\end{remark}

\

\begin{center}
\textbf{Acknowledgements}
\end{center}
The authors would like to thank Prof. P.L. Yung for the comments concerning the topic of this paper. The research of the first author is supported by the grant of thousand youth talents plan of China. The research of the second author is partially supported by PRIN12 project: \emph{Variational and Perturbative Aspects of Nonlinear Differential Problems} and FIRB project: \emph{Analysis and Beyond}.

\


\begin{thebibliography}{99}

\bibitem{a}
R. Abazari, The $(\frac{G'}{G})$-expansion method for Tzitz\'eica type nonlinear evolution equations, \emph{Math. and Comp. Model.} 52 (2010) 1834-1845.

\bibitem{ajy}
W. Ao, A. Jevnikar, W. Yang, On the boundary behavior for the blow up solutions of the sinh-Gordon equation and rank $N$ Toda systems in bounded domains, preprint, 2017, http://cvgmt.sns.it/paper/3215/.

\bibitem{ba} A. Bahri and J.M. Coron, The scalar curvature problem on the standard three dimensional sphere, \emph{J. Funct. Anal.} 95 (1991), no. 1, 106-172.

\bibitem{b}
L. Battaglia, $B_2$ and $G_2$ Toda systems on compact surfaces: a variational approach, \emph{J. Math. Phys} 58 (2017), no. 1, 011506.

\bibitem{bjmr} L. Battaglia, A. Jevnikar, A. Malchiodi, D. Ruiz, A general existence result for the Toda system on compact surfaces. {\em Adv. Math.} 285 (2015), 937-979.

\bibitem{bw}
J. Bolton, L.M. Woodward, Some geometrical aspects of the 2-dimensional Toda equations. In Geometry, topology and physics (Campinas, 1996), pages 69-81. de Gruyter, Berlin (1997).

\bibitem{cm}
I. Cabrera-Carnero, M. Moriconi, Noncommutative Integrable Field Theories in 2d, \emph{Nucl. Phys. B} 673 (2003), 437-454.

\bibitem{cag}
E. Caglioti, P.L. Lions, C. Marchioro, M. Pulvirenti, A special class of stationary flows for two-dimensional Euler equations: a statistical mechanics description, \emph{Comm. Math. Phys.} 143 (1992), no. 3, 501-525.

\bibitem{cal}
E. Calabi, Isometric imbedding of complex manifolds, \emph{Ann. of Math. (2)}, 58 (1953), 1-23.

\bibitem{cy0} S.Y. Chang, P.C. Yang,  Prescribing Gaussian curvature on $S^2$. {\em Acta Math.}  159  (1987),  no. 3-4, 215-259.

\bibitem{cy} S. Chanillo, P.L. Yung, Wave equations associated to Liouville systems and Constant Mean Curvature equations, {\em Adv. Math} 235 (2013), 187-207.

\bibitem{ch}
S.S. Chern, Geometrical interpretation of the sinh-Gordon equation, \emph{Ann. Pol. Math.} 39  (1980), 74-80.

\bibitem{c} A.J. Chorin, Vorticity and Turbulance, Springer, New York (1994).

\bibitem{cho}
K.W. Chow, A class of doubly periodic waves for nonlinear evolution equations, \emph{Wave Motion}, 35 (2002), 71-90.

\bibitem{d1} G. Dunne, Self-dual Chern-Simons Theories. Lecture Notes in Physics. Springer, Berlin (1995).

\bibitem{fll}
Z. Fua, S. Liu, S. Liu, Exact Solutions to Double and Triple Sinh-Gordon Equations, \emph{Zeit. Natur. A} 59 (2004), 933-937.


\bibitem{gjm} C. Gui, A. Jevnikar, A. Moradifam, Symmetry and uniqueness of solutions to some Liouville-type equations and systems, to appear in \emph{Comm. PDEs.}

\bibitem{j1} A. Jevnikar, An existence result for the mean field equation on compact surfaces in a doubly supercritical regime, \emph{Proc. Royal Soc. Edinb. A} 143 (2013), no. 5, 1021-1045.

\bibitem{j2} A. Jevnikar, Multiplicity results for the mean field equation on compact surfaces, \emph{Adv. Nonlinear Stud.} 16 (2016), no. 2, 221-229.

\bibitem{j3} A. Jevnikar, New existence results for the mean field equation on compact surfaces via degree theory. \emph{Rend. Semin. Mat. Univ. Padova} 136 (2016), 11-17 .

\bibitem{j4}
A. Jevnikar, Blow-up analysis and existence results in the supercritical case for an asymmetric mean field equation with variable intensities, \emph{J. Diff. Eq.} 263 (2017), 972-1008.

\bibitem{jkm}
A. Jevnikar, S. Kallel, A. Malchiodi, A topological join construction and the Toda system on compact surfaces of arbitrary genus, \emph{Anal. PDE} 8 (2015), no. 8, 1963-2027.

\bibitem{jwy} A. Jevnikar, J.C. Wei, W. Yang, Classification of blow-up limits for the sinh-Gordon equation, to appear in {\em Differential and Integral Equations}.

\bibitem{jwy2} A. Jevnikar, J. Wei, W. Yang, On the Topological degree of the Mean field equation with two parameters, to appear in {\em Indiana Univ. Math. J.}

\bibitem{jy}
A. Jevnikar, W. Yang, Analytic aspects of the Tzitz\'eica equation: blow-up analysis and existence results, \emph{Calc. Var. and PDEs} 56 (2017), no. 2, 56:43.

\bibitem{jlw}
J. Jost, C.S. Lin, G. Wang, Analytic aspects of the Toda system. II. Bubbling behavior and existence of solutions, \emph{Comm. Pure Appl. Math.} 59 (2006), no. 4, 526-558.

\bibitem{jw} J. Jost, G.F. Wang, Analytic aspects of the Toda system: I. A Moser-Trudinger inequality. {\em Comm. Pure Appl. Math.} 54 (2011), no. 11, 1289-1319.

\bibitem{jwyz} J. Jost, G. Wang, D. Ye, C. Zhou, The blow-up analysis of solutions of the elliptic sinh-Gordon equation, {\em Calc. Var. PDEs} 31 (2008), 263-276.

\bibitem{jm} G. Joyce, D. Montgomery, Negative temperature states for a two dimensional guiding center plasma, {\em J. Plasma Phys.} 10 (1973), 107-121.

\bibitem{kies}
M.K.H. Kiessling, Statistical mechanics of classical particles with logarithmic interactions, \emph{Comm. Pure Appl. Math.} 46 (1993), no. 1, 27-56.

\bibitem{lwy} C.S. Lin, J.C. Wei, W. Yang, Degree counting and shadow system for SU(3) Toda system: one bubbling, preprint, 2014, arXiv http://arxiv.org/abs/1408.5802.

\bibitem{lwy-c}
C.S. Lin, J.C. Wei, D. Ye, Classification and nondegeneracy of $SU(n+1)$ Toda system, \emph{Invent. Math.} 190 (2012), no. 1, 169-207.

\bibitem{l}  P.L. Lions, On Euler Equations and Statistical Physics, Scuola Normale Superiore, Pisa (1997).

\bibitem{m} A. Malchiodi, Topological methods for an elliptic equation with exponential nonlinearities. {\em Discrete Contin. Dyn. Syst.} 21 (2008), no. 1, 277-294.

\bibitem{mr}
A. Malchiodi, D. Ruiz, A variational Analysis of the Toda System on Compact Surfaces, \emph{Comm. Pure Appl. Math.} 66 (2013), no. 3, 332-371.

\bibitem{mp} C. Marchioro, M. Pulvirenti, Mathematical theory of incompressible nonviscous fluids, Springer, New York (1994).

\bibitem{mmr}
P. Mosconi, G. Mussardo, V. Rida, Boundary Quantum Field Theories with Infinite Resonance States, \emph{Nucl. Phys. B} 621 (2002), 571-586.

\bibitem{nat} F. Natali, On periodic waves for sine- and sinh-Gordon equations, \emph{J. Math. Anal. Appl.} 379 (2011), 334-350.

\bibitem{n} P.K. Newton, The $N$-Vortex Problem: Analytical Techniques, Springer, New York (2001).

\bibitem{os} H. Ohtsuka,T. Suzuki, Mean field equation for the equilibrium turbulence and a related functional inequality, {\em Adv. Differential Equations} 11 (2006), 281-304.

\bibitem{ons}
L. Onsager, Statistical hydrodynamics, \emph{Nuovo Cimento Suppl.} 6 (1949), 279-287.

\bibitem{pl} Y.B. Pointin, T.S. Lundgren, Statistical mechanics of two-dimensional vortices in a bounded container, {\em Phys. Fluids} 19 (1976), 1459-1470.

\bibitem{ss}
K. Sawada, T. Suzuki, Derivation of the equilibrium mean field equations of point vortex and vortex filament system, \emph{Theoret. Appl. Mech. Japan} 56 (2008), 285-290.

\bibitem{sch}
R. Schoen, D. Zhang, Prescribed scalar curvature on the n-sphere, \emph{Calc. Var.} 4 (1996), no. 1, 1-25

\bibitem{sw}
I. Shafrir, G. Wolansky, Moser-Trudinger and logarithmic HLS inequalities for systems, \emph{J. Eur. Math. Soc. (JEMS)} 7 (2005), no. 4, 413-448.

\bibitem{tar}
G. Tarantello, Analytical, geometrical and topological aspects of a class of mean field equations on surfaces, \emph{Discrete Contin. Dyn. Syst.} 28 (2010), no. 3, 931-973

\bibitem{tar2}
G. Tarantello, Selfdual gauge field vortices: an analytical approach. Progress in Nonlinear Differential Equations and their Applications, 72. Birkh\"auser Boston Inc., Boston, MA (2008).

\bibitem{w}
A.M. Wazwaz, Exact Solutions to the Double Sinh-Gordon Equation by the Tanh Method and a Variable Separated ODE Method, \emph{Comp. Math. with Appl.} 50 (2005), 1685-1696.

\bibitem{w1} H.C. Wente, Large solutions to the volume constrained Plateau problem, {\em Arch. Rational Mech. Anal.} 75 (1980/81), no. 1, 59-77.

\bibitem{w2} H.C. Wente, Counterexample to a conjecture of H. Hopf, {\em Pacific J. Math.} 121 (1986), no. 1 , 193-243.

\bibitem{y} Y. Yang, Solitons in Field Theory and Nonlinear Analysis. Springer Monographs in Mathematics. Springer, New York (2001).

\bibitem{zbp}
W.P. Zhong, M.R. Beli\'c, M.S. Petrovi\'c, Solitary and extended waves in the generalized sinh-Gordon equation with a variable coefficient, \emph{Nonlinear Dyn.} 76 (2014), 717-723.



\end{thebibliography}
\end{document}